\documentclass[11pt]{amsart}
\usepackage{epsfig}

\usepackage{amsmath}
\usepackage{marvosym}
\usepackage{fancyhdr}
\usepackage{amssymb}
\usepackage[utf8]{inputenc}
\usepackage{amsthm}
\usepackage{epsfig}
\usepackage{graphicx}
\usepackage[pagewise]{lineno}

\newcommand{\RR}{\rm I\kern -1.6pt{\rm R}}

\newtheorem{theorem}{Theorem}[section]
\newtheorem{lemma}{Lemma}[section]
\newtheorem{assumption}{Assumption}[section]

\newtheorem{definition}{Definition}[section]

\newtheorem{corollary}{Corollary}[section]

\numberwithin{equation}{section} %\numberwithin{figure}{section}
\begin{document}

\date{}
\title[The period function for potential system]
{Monotonicity and critical points of the period function for potential system}
\author{Jihua Wang*}
\thanks{*Corresponding author:~wangjh2000@163.com (J. Wang)}
\address{}

\maketitle
{\bf Abstract:} This paper is concerned with the analytic
behaviors (monotonicity, isochronicity and the number of critical
points) of period function for potential system $\ddot{x}+g(x)=0$.
We give some sufficient criteria to determine the monotonicity and
upper bound to the number of critical periods. The conclusion is based on the semi-group properties of (Riemann-Liouville) fractional integral operator of order $\frac{1}{2}$ and Rolle's Theorem. In polynomial potential settings, bounding the the number of critical periods of potential center can be reduced to counting the real zeros of a semi-algebraic system. From which we prove that if nonlinear potential $g$ is odd, the potential center has at most $\frac{deg(g)-3}{2}$ critical periods. To illustrate its applicability some known results are proved in more efficient way,
and the critical periods of some hyper-elliptic Hamiltonian systems of
degree five with complex critical points are discussed, it is proved
the system can have exactly two critical periods.



\section{Introduction and statement of the results}
This paper is concerned with the period of periodic solutions of
conservative second order equation $\ddot{x}+g(x)=0$ or equivalently
the period function of centers of planar potential system with the
form
\begin{equation}\label{SN}
\begin{split}
\frac{dx}{dt}=-y, \quad \frac{dy}{dt}&=g(x).
\end{split}
\end{equation}
The system has one degree of freedom``kinetic+potential" Hamiltonian
\begin{equation}\label{HAMFUN}
H(x,y)=\frac{y^2}{2}+G(x)=h,\quad
\end{equation}
where $G(x)=\int_0^xg(s)ds$ is a primitive function of $g(x)$ which
denotes the potential energy of this classical mechanical system.

Recall that a singular point of planar differential system is a
center if it has a punctured neighborhood foliated with a continuous
band of periodic orbits $\Gamma_h$ which is determined by equation
\eqref{HAMFUN}. The largest punctured neighborhood with this
property is called period annulus, denoted by $\mathcal{P}$. It is
obvious that the vector field generated by system \eqref{SN} is
symmetric with respect to $x-$axis, we denote the projection of
$\mathcal{P}$ onto $x-$axis by $(x_m,x_M)$. Parameterizing the
periodic orbits by energy level $h$ and assigning the minimal
positive period of the motion along each of the periodic orbit
$\Gamma_h$, then we have a energy-period function
$T:\,(0,h_s)\mapsto\,(0,+\infty)$. As $g$ is a polynomial, the
period of $\Gamma_h$ is given by the following Abelian integral
\begin{equation}\label{EXPT1}
T(h)=-\oint_{\Gamma_h}\frac{dx}{y}=\sqrt{2}\int_{x_{-}(h)}^{x_{+}(h)}\frac{dx}{\sqrt{h-G(x)}},
\end{equation}
where $x_{-}(h)<0<x_{+}(h)$ are the abscissas of the intersection
points between $\Gamma_h$ and $x-$axis. We say $T$ is monotone
increasing (or decreasing) if for any couple of periodic orbits
$\Gamma_{h_1}$ and $\Gamma_{h_2}$ in $\mathcal{P}$, there holds
$T(h_1)<T(h_2)$ for $h_1<h_2.$ (resp. $T(h_1)>T(h_2)$). The local
maximum or minimum of the period function is called critical period,
the periodic orbit with this property is called critical periodic
orbit \cite{JVXZ}. If the period of oval $\Gamma_h$ is independent
with the energy level $h$, namely $T(h)$ is constant, then the
center is called isochronous, see the survey paper \cite{JCMA}.

Aside from their intrinsic interest, the analytical properties
(monotonicity, isochronicity  or critical points) of period function
arise from the study of oscillation features such as the soft or
hard springs, also occur in the study of subharmonic bifurcations of
periodic oscillations and the bifurcation of steady-state solutions
of a reaction-diffusion equation, in the problem on the existence
and uniqueness of solutions in boundary value problem
\cite{WSL,PBSC}. In some sense the study of critical periodic orbits
is analogous to the study of limit cycles \cite{CCMJ,MGJV},
which is the main concern of the celebrated Hilbert's 16th problem
and its various weakened versions.

In what follows, we suppose that the annulus $\mathcal{P}$ only
surrounds the center and no other equilibria, it follows the
equality $G(x_m)=G(x_M)$ holds. Without loss of generality, we
assume the center of system \eqref{SN} locates at the origin, it is
known that the multiplicity of $x=0$ as the zero of $g(x)$ must be
odd. Therefore we assume
\begin{assumption}\label{PROPG}
The potential $g(x)$ is smooth enough in $(x_m, x_M)$ and $k$ is
non-negative integer.
\begin{enumerate}
  \item $g(0)=g^{\prime}(0)=\cdots=g^{(2k)}(0)=0,\; g^{(2k+1)}(0)>0$
   and $xg(x)>0$ for all $x\in\,(x_m,x_M)\setminus\,\{0\}$;
  \item $G(x_m)=G(x_M)=h_s$ for $h_s\in(0,+\infty].$
\end{enumerate}
\end{assumption}

Under the Assumption \ref{PROPG}, if $k=0$ the origin is an
elementary (or nondegenerate) center for the linear part of the
singular point has a pair of conjugate pure imaginary eigenvalues
(or the Jacobian matrix at which is nondegenerate, i.e., its
determinant does not vanish). If $k\geq\,1$ system \eqref{SN} at the origin $O$ has nilpotent linear part
$-y\frac{\partial}{\partial\,x}$, accordingly it is called a
nilpotent center. If the period annulus $\mathcal{P}$ is unbounded,
the center is said to be global. Hamiltonian \eqref{HAMFUN} has a
local minimum $H(0,0)=0$ at the origin, we assert that
$G(x)=\frac{A}{2}x^{2k+2}+h.o.t$ and $A>0$, thus there exists an
analytic involution $\sigma$ on $(x_m,\,x_M)$ such that
\begin{equation}\label{INVOL}
G(x)=G(z),\,z=\sigma(x)=-x+o(x), \; x\in\,(x_m,x_M)\backslash\{0\}.
\end{equation}
Recall that an involution $\sigma$ is a diffeomorphism with a unique
fixed point satisfying $\sigma\circ \sigma$=Id and $\sigma\neq\,$Id.
Following the terminology proposed in \cite{MFJ}, we define the
balance of function $f$ with respect to the involution $\sigma$ as
\[\mathcal{B}_{\sigma}(f)(x)=f(x)-f(z).\]

we can state the main results of
the present paper.
\begin{theorem}\label{PFCP}
Consider the potential system \eqref{SN}, suppose $g$ is smooth in
$(x_m,x_M)$ and satisfies Assumption \ref{PROPG}, for all
$x\in\,(0,x_M)$,
\begin{enumerate}\label{PFDISCR}
\item  if $\mathcal{B}_{\sigma}(\delta)(x)>0$\,($<0$, resp.) then the period function $T(h)$ is monotone
increasing (or decreasing, resp.) in $(0,h_s)$;
\item  if $\mathcal{B}_{\sigma}(\delta)(x)$ has $l$ zeros, then system \eqref{SN} has at most $l$
critical periods; moreover if $\mathcal{B}_{\sigma}(\frac{G}{g^2})(x)$ has $l$ zeros as well, then the system \eqref{SN} has exactly $l$
critical periods;
\item if $\mathcal{B}_{\sigma}(\delta)(x)\equiv\,0$, then the origin of
system \eqref{SN} is an isochronous center.
\end{enumerate}
\end{theorem}

where
\[\delta(x)=\big(\frac{G}{g^2})^{\prime}(x)=\frac{g^2-2Gg^{\prime}}{g^3}(x).\]

Note that the function $\frac{G}{g^2}(x)$ and its derivative
$(\frac{G}{g^2})^{\prime}(x)$ play an important roll in the study of
the period function, which appear in many works in this field, see
Loud \cite{WSL}, Coppel and Gavrilov \cite{CG}, Chicone \cite{CCE}, Zeng and Jing \cite{XZZJ}, A. Cima et al. \cite{CGS}, Li and Lu \cite{CLKL}, Yang and Zeng \cite{LYXZ}, Villadelprat and Zhang \cite{JVXZ} and the references therein. If $k=0$ from the Assumption
\ref{PROPG}, it can be derived that $\delta(x)$ is continuous in
$(x_m,x_M)$ and continuously differentiable in $(x_m,0)$ and
$(0,x_M).$ In fact, by applying L'Hospital's rule and Assumption
\ref{PROPG}, we get
\begin{equation}\label{DELEO}
\delta(0)=-\frac{1}{3}\frac{g^{\prime\prime}}{(g^{\prime})^2}(0),\quad
\delta^{\prime}(0)=\frac{5(g^{\prime\prime})^2-3g^{\prime}g^{\prime\prime\prime}}{12(g^{\prime})^3}(0).
\end{equation}
If $k\geq\,1$, it is easy to deduce from Assumption \ref{PROPG} that
\begin{equation}\label{DELNO}
\delta(0^+)=-\infty,\quad \delta(0^-)=\infty.
\end{equation}
Note that the outer boundary of period annulus must contain singular
points. Thus from Assumption \ref{PROPG}, if $g(x_m)g(x_M)=0$
therefore
\begin{equation}\label{DELWB}
\mathcal{B}_{\sigma}(\delta)(x_M)=\delta(x_M)-\delta(x_m)=\infty.
\end{equation}

Consider potential system $\ddot{x}+g(x)=0$, the authors \cite{PBSC}
proved for generic $g$, the period function is morse function, which means
all the critical points of $T$ are nondegenerate. If the potential $g$ is a polynomial in $x$,
it is proved in \cite{CCMJ} that the unique polynomial isochronous center of
potential system is the linear and global one. Recall that only nondegenerate center
can be isochronous. C. Chicone, F. Dumortier proved in Theorem 1.2 of \cite{CCFD} that the Hamiltonian
system \eqref{SN} with polynomial potential $g$ of degree two or
more has finite number of critical periods. Chow and Sanders
\cite{SCJS} ask whether there exists a bound on the number of
critical periods and assume that the bound if it exists, depends on
the degree of the polynomial $g(x)$. In what follows, we propose a
positive answer to the problem. By applying Theorem \ref{PFCP}
together with Rolle's Theorem and its generalization, Theorem \ref{GRTHEM}, we state the second main results as follows
\begin{theorem}\label{NCPSN}
Consider the potential system \eqref{SN} with polynomial potential
$g$ of degree $\geq\,2$, under the assumption \eqref{PROPG}, we
have
\begin{enumerate}
\item if the potential $g$ is odd function, thus the system
\eqref{SN} has at most $\frac{deg(g)-3}{2}$ critical periods;

\item if all the zeros of $g$ are real, then system \eqref{SN} has at most one critical period.
\end{enumerate}
\end{theorem}

It is worth mentioning that the authors \cite{CCMJ} give even potential conjecture that if the potential energy $G(x)$ is even, the center of the system at the origin has at most $\frac{deg(g)-3}{2}$ local
critical periods. Here we confirm the result holds globally.

To ensure the monotonicity of period function, i.e., the absence of
critical periodic orbits,  from Theorem \ref{PFCP}(1), it suffices
to show
\[\mathcal{B}_{\sigma}(\delta)(x)=\delta(x)-\delta(\sigma(x))>0,(<0,resp,)\]
where $x_m<\sigma(x)<0<x<x_M$. Hence it is easy to get Chicone's
criterion \cite{CCE}
\begin{corollary}\label{COR1}
If $\delta^{\prime}(x)>0$ (or $\delta^{\prime}(x)<0$, respectively)
in $(x_m,x_M)$, then $T(h)$ is monotone increasing (or decreasing,
respectively) in $(0,h_s)$.
\end{corollary}
This criterion provides a sufficient condition for monotonicity
which is easy to verify, while its disadvantages is that the
condition is far away from being necessary, as a generalization see Theorem 1 and
corollary 1 of \cite{XZZJ} for instance.

One necessary condition on the
monotonicity of period for the potential center \eqref{SN} was stated in Theorem 2.6 of \cite{FMJV1} as follows
\begin{lemma}\label{MFNC}
If the period function $T(h)$ of system \eqref{SN} is monotone in $(0,h_s)$, then the function
$x\mapsto\,\frac{G(x)}{(x-\sigma(x))^2}$ is monotone in $(0,x_s),$ where $G(x_s)=h_s$ corresponds to
the energy level of outer contour of period annulus.
\end{lemma}

In literatures it is noted in \cite{CCFD} that any given analytic
vector field of center type has a finite number of critical periods
on a period annulus contained in a compact region. To our knowledge,
consider a family of centers with critical periods, the method often
applied therein bases on proposition where the period function satisfies some
kind of Picard-Fuchs differential equation for algebraic curves
\cite{LG,FMJV}. This approach seems hard to be used in the cases of
polynomial potential $g$ with high degree or non-polynomial
settings. Cima et al.,\cite{CGS} obtain some lower bounds for the
number of critical periods of the families of potential, reversible
and Li\'{e}nard centers by perturbing linear one. Concerning the
Hamiltonian systems with separable variables \cite{MS,JVXZ}, the
results there exists at most one critical period are
proved by means of studying the period function's convexity.
Ma\~{n}osas and Villadelprat \cite{FMJV1} propose a criterion to
bound the number of critical periods, while the criterion function
is defined recursively which increases computational complexity
rapidly.

Theorem \ref{PFCP}(2) proposes an efficient
criterion for (sharp) upper bound to the number of critical periods,
which can be reduced to counting the number of zeros of balance
function. If potential $g$ is polynomial, it reduces to solving the
following semi-algebraic system (SAS for short).
\begin{equation}\label{SAS}
G(x)=G(z),\;\delta(x)=\delta(z)
\end{equation} in the domain
\[0<x<X_M,\;x_m<z<0.\]
By SAS, we mean the system consists of polynomial equations,
polynomial inequations and inequalities \cite{YL}. There are various
problems in both practice and theory which need to solving a
semi-algebraic system. A SAS formally can be represented as $[F, N,
P, H](U, X)$ where $U=(u_1,u_2,\cdots,u_d)$ is vector of parameters,
and $X=(x_1,x_2,\cdots,x_n)$ is vector of variables. $F,N,P$ and $H$
stands for the equations $[f_1,f_2,\cdots,f_s]$, non-negative
inequalities $[g_1,g_2,\cdots,g_r]$, positive inequalities
$[g_{r+1},\cdots,g_t]$ and inequations $[h_1,h_2,\cdots,h_m]$,
respectively, where $f_i,g_j,h_k\in\,Z(U)[X]$ are all polynomials
with integer coefficients with subordinates $d,r,t,m\geq\,0$ and
$n,s\geq\,1$. If $d=0$, the SAS is called constant. If $d>0$, the
SAS is called parametric, see Appendix A of \cite{JW1} for more
details.

In the viewpoint of symbolic computation, how to isolate and count
the distinct roots of a SAS are two crucial issues. By using command
{\bf RealRootClassification} of the package {\bf
SemiAlgebraicSetTools} in computer algebra system Maple 13 or its
higher versions, one can solve the SAS with parameters, Parametric
SAS for short. {\bf RealRootClassification} computes conditions on
the parameters for the system to have a given number of real
solutions. The algorithm implemented by Maple is described in
\cite{YLHXXB}. The algorithm is complete in theory, however, the
command {\bf RealRootClassification} may quit in some case due to
too heavy computation. If the parameters are assigned, the constant
SAS can be easily solved by the command {\bf RealRootCounting} of
the package {\bf SemiAlgebraicSetTools} which returns the number of
distinct real solutions. The corresponding algorithm is described in
the paper coauthored by Xia and Hou \cite{XBHX}. In order to know
exactly the distinct real roots, the {\bf RealRootIsolate} command
returns a list of boxes. Each box isolates exactly one real root of
semi-algebraic system. All the calling sequence, descriptions and
examples of command can be viewed by consulting ``?Command name" in
programme interface of Maple.

Alternatively one can seek some (explicit) conditions on $g$ so that
$\mathcal{B}_{\sigma}(\delta)(x)$ is non-vanishing or has zeros. To
this end Rolle's Theorem or Intermediate-value Theorem are often
applied to count the number of zeros of univariate functions. To
solve the SAS \eqref{SAS} in a geometric way, it need to count the
number of intersections of two planar smooth curves, thus the following
generalized Rolle's Theorem (See Theorem 6.11 in \cite{HZCR}
for instance) is useful.

\begin{theorem}\label{GRTHEM}
(Generalized Rolle's Theorem.) Let $D=(x_l, x_r)\times(y_b, y_t)$.
Let $F(x,y),G(x,y)$ be two functions continuous on closure $\bar{D}$
and smooth in $D$. Suppose that $F^{\prime}_x(x,y)\neq\,0,\;
F^{\prime}_y(x,y)\neq\,0$ in $D$. Consider the number of solution of
system
\[F(x,y)=0,\; G(x,y)=0,\quad (x,y)\in\,D.\]
Denote the number by $\sharp\{(F,\,G)|(x,y)\in\,D\}$ and let the
Jacobian
\[J[F,G](x,y)=F^{\prime}_x(x,y)G^{\prime}_y(x,y)-F^{\prime}_y(x,y)G^{\prime}_x(x,y),\]
then
\[\sharp\{(F,G)|(x,y)\in\,D\}\leq\,1+\sharp\{(F,J[F,G])|(x,y)\in\,D\}.\]
\end{theorem}

The rest part of paper is organized as follows. In section \ref{FRACID} we consider the properties of (Riemann-Liouvile) fractional derivative with which to prove Theorem \ref{PFCP}(2), where properties of Variation Diminishing of Abel's integral operator or fractional calculus with the order $\frac{1}{2}$ is given. This is the highlight of present paper. In Section
\ref{ASYTH} some asymptotical results of period function near the
center and the outer boundary are proposed, together with Theorem
\ref{PFCP}, the global analytical behaviors of period function can
be determined. The proof of main results of present paper, Theorem
\ref{PFCP} is given in Section \ref{PHMR}, the main idea is to
deduce the formula \eqref{EXPTD13} about $T^{\prime}(h)$ which is formulated as the integral of Balance function with order $\frac{1}{2}$.
As the applications of Theorem \ref{PFCP}, in Section \ref{PFRZ}
some known results are proved in more efficient way. To the best of
our knowledges, there are few works on existence of multiple critical
periods in literatures. The period function of some hyperelliptic
Hamiltonians of degree five with complex critical points are
considered, it is proved the system can have two critical periods.
Some comments are stated in Section \ref{DISS}.

\section{Fractional calculus and properties of Variation Diminishing of Abel's integral operator}\label{FRACID}
On account of the upper bound to the number of critical periods of
potential center \eqref{SN}, Theorem 3.7 in \cite{LYXZ} takes
advantage of Zhang's differential operator to deduce the higher
order derivative of period function. Theorem A of \cite{FMJV1} bases
on the Chebyshev properties of the tuple of criterion functions
defined recursively. In this section we intend to consider an
integral operator with fractional order with which to prove Theorem \ref{PFCP}(2).

Recall that in studying problem of Isochrone N. H. Abel at 1823 derived Abel's integral equation
\begin{equation}\label{INTOP}
A(k)(h):=\frac{1}{\Gamma{(\frac{1}{2}})}\int_0^h\frac{k(x)}{\sqrt{h-x}}dx,\quad \forall
k\in\,C^0[0,x_M].
\end{equation}
where $A(k)(h)$ denotes the Riemann-Liouville integral of function $k(x)$ with order $\frac{1}{2}.$
In the viewpoint of inverse problem, it is a singular Volterra integral equation of convolution type with weakly singular integral kernel $\frac{1}{\sqrt{h-x}}$. If free term $A(k)(h)$ is continuously differentiable and vanishes at $h=0$ then
equation \eqref{INTOP} has exactly one continuous solution formed as
\begin{equation}\label{SOLN}
k(x)=\frac{1}{\Gamma{(\frac{1}{2}})}\cdot\frac{d}{dx}\int_0^x\frac{A(k)(t)}{\sqrt{x-t}}dt=\frac{1}{\sqrt{\pi}}\big[\frac{A(k)(0)}{\sqrt{x}}+\int_0^x\frac{A(k)_t^{\prime}(t)}{\sqrt{x-t}}dt\big],\;x\in[0,x_M).
\end{equation}
Note that $A(k)(h)$ vanishes at $h=0$ is the necessary condition to
the existence of solution \eqref{SOLN} at $x=0$.

Fractional calculus \cite{OS} used to denote calculus of non-integer order which is in contrast to the classical calculus as created independently by Newton and Leibniz. We recall the definitions of fractional integral and derivative as follows \cite{OS, IP}.

\begin{definition}
Suppose $\alpha>0$ and $f(x)\in\,L^1(0,+\infty)$. Then for all $x>0$ we call
\begin{equation}\label{RLI}
I_{0}^{\alpha}(f)(x)=\frac{1}{\Gamma(\alpha)}\int_0^x\frac{f(s)}{(x-s)^{1-\alpha}}ds
\end{equation}
the (below) Riemann-Liouville integral of $f$ of order $\alpha$.
\end{definition}

\begin{definition}
Suppose $\alpha>0$ and suffices $n-1\leq\,\alpha<n$ where $n$ is a positive integer. if $f(x)\in\,C^n(R)$ we call
\begin{equation}\label{RLD}
D_{0}^{\alpha}(f)(x)=\frac{1}{\Gamma(n-\alpha)}\frac{d^n}{dx^n}\int_0^x\frac{f(s)}{(x-s)^{\alpha-n+1}}ds
\end{equation}
the (below) Riemann-Liouville derivative of $f$ of order $\alpha$.
\end{definition}

Remark that the order of fractional differ-integration can be rational, irrational or complex number. When order $\alpha$ is a positive integer, the differ-integration \eqref{RLI} and \eqref{RLD} are in accordance with the classical case. Noting that fractional derivatives are non-local in contrast to the classical
derivative. For example, the fractional derivative of constant function is not zero. Riemann--Liouville derivative depends on a free parameter which relies on global information of the function. Nevertheless the (below) Riemann-Liouville fractional differential operator still satisfies the following classical proposition \cite{IP}.
\begin{lemma}\label{FLEM}
The (below) Riemann-Liouville fractional integral operator is commutable, suppose that function $f$ is continuous in $(a,b)$ for any $\mu,\nu>0$ there holds
\[I_{a}^{\mu}I_{a}^{\nu}f(x)=I_{a}^{\mu+\nu}f(x)=I_{a}^{\nu}I_{a}^{\mu}f(x).\]
\end{lemma}
\begin{proof}
By Definition \ref{RLI} and change the order of double integral we get
\begin{eqnarray*}
I_{a}^{\mu}I_{a}^{\nu}(f)(t)&=&\frac{1}{\Gamma(\mu)\Gamma(\nu)}\int_a^t\,(t-\tau)^{\mu-1}\int_a^\tau\,f(s)(\tau-s)^{\nu-1}dsd\tau\\
                            &=&\frac{1}{\Gamma(\mu)\Gamma(\nu)}\int_a^t\,f(s)\int_s^t\,(t-\tau)^{\mu-1}(\tau-s)^{\nu-1}d\tau\,ds
\end{eqnarray*}
By applying change of coordinates $\xi=(\tau-s)/(t-s)$ and proposition of Beta function, furthermore we obtain
\begin{eqnarray*}
I_{a}^{\mu}I_{a}^{\nu}(f)(t)&=&\frac{1}{\Gamma(\mu)\Gamma(\nu)}\int_a^t\,(t-s)^{\mu+\nu-1}f(s)\int_0^1\,(1-\xi)^{\mu-1}\xi^{\nu-1}d\xi\,ds\\
                            &=&\frac{\mathbf{B}(\mu,\nu)}{\Gamma(\mu)\Gamma(\nu)}\int_a^t\,(t-s)^{\mu+\nu-1}f(s)\,ds\\
                            &=&I_{a}^{\mu+\nu}f(t).
\end{eqnarray*}
where $\Gamma(\cdot),\,\mathbf{B}(\cdot,\cdot)$ represent gamma function and Beta function, respectively.
\end{proof}

As in classical calculus, we have the equality $D_a^\alpha[{}I_a^\alpha\,f(x)]=f(x)$ for any $\alpha>0$ as well,  namely the derivative operator $D_a^\alpha$ is the left inverse of $I_a^\alpha$. It is obvious $I_a^\alpha$ is a positive linear operator and the set $\{I_a^\alpha,\alpha>0\}$ is formed as strongly continuous semi-group with respect to $\alpha$.

From the Lemma \ref{FLEM} we assert the following results.

\begin{theorem}\label{PVD}
Abel's integral operator \eqref{INTOP} is Variation Diminishing. Suppose that function $k\in\,L^1(0,h_s)$,
then the number of zeros of function $k$ in $(0,h_s)$ is an upper bound to that of image function $A(k)(h)$ (i.e.,the Riemann-Liouville integral of $k(x)$ with order $\frac{1}{2}$); moreover the bound is sharp when the primitive function of $k$ has the same number of zeros as $k$ in $(0,h_s)$.
\end{theorem}
\begin{proof}
We prove the conclusion by contradiction. Denote the number of zeros of function $f$ in the interval $(a,b)$ by
$\sharp\,f|_{(a.b)}.$ Assume that
\[\sharp\,A(k)|_{(0,h_s)}>\sharp\,k|_{(0,h_s)}. \]
Without of generality say $\sharp\,A(k)|_{(0,h_s)}=\sharp\,k|_{(0,h_s)}+1.$ Note that $I_0^1(k)=A\circ\,A(k)$ is the primitive function of $k$.
By assumption we have
\[\sharp\,I_0^1(k)|_{(0,h_s)}\geq\sharp\,k|_{(0,h_s)}+2,\]
which contradicts the classical Rolle's Theorem. It leads to the first conclusion.

From the above discussions, we have
\[\sharp\,I_0^1(k)|_{(0,h_s)}\leq\,\sharp\,A(k)|_{(0,h_s)}\leq\sharp\,k|_{(0,h_s)},\]
The second result follows.
\end{proof}

\section{The asymptotical results of period function near the
boundary}\label{ASYTH}

Note that Theorem \ref{PFDISCR}(2) proposes an practical approach to
bound the number of critical periods. If taking into account the
monotonicity of period function near the two endpoints of interval
$(0,h_s)$, one can determine analytical behaviors of $T(h)$ globally
in $(0,\,h_s)$. On the Poincar\'{e}-Lyapunov disc, the period
annulus of center has two connect components, the inner one is
center itself and the outer boundary is a polycycle.

The first terms of the Taylor expansions of period function near the
nondegenerate center is well known, see Proposition 3.5 in
\cite{LYXZ} and formulas (25a) and (25b) of \cite{RF}, Theorem 4.2
\cite{CCFD} for instance. For completeness we state as follows.
\begin{lemma}\label{PHDPHO}
Assume function $g(x)$ of potential system \eqref{SN} is analytic
and satisfies the Assumption \ref{PROPG}.

If $k=0$, then period function \eqref{EXPT1} is analytic in
$(0,h_s)$ and can be analytically extended to $h=0$. Moreover
function $T(h)$ and its derivation $T^{\prime}(h)$ as
$h\rightarrow\,0^+$ satisfy the following limits
\begin{equation}
         T(0^+)=\frac{2\pi}{\sqrt{g^{\prime}(0)}},\;
T^{\prime}(0^+)=\pi\frac{{5(g^{\prime\prime}(x))^2-3\,g^{\prime}(x)g^{\prime\prime\prime}(x)}}{12(g^{\prime}(x))^{\frac{7}{2}}}|_{x=0}.
\end{equation}
If $k\geq\,1$, then
\begin{equation}
         T(0^+)=\lim\limits_{h\rightarrow\,0^+}h^{-\frac{k}{2k+2}}\lambda_k=+\infty,\quad T^{\prime}(0^+)=\lim\limits_{h\rightarrow\,0^+}-\frac{k\lambda_kh^{-\frac{3k+2}{2k+2}}}{2k+2}=-\infty.
\end{equation}
where $\lambda_k$ is a positive constant.
\end{lemma}

Concerning the limit of the period function $T(h)$ as $h$ tends to energy level $h_s$ corresponding to the outer boundary of
period annulus $\Gamma_s$, we discuss in two cases.

If $h_s<+\infty$ which means the outer boundary $\Gamma_s$ is a
polycycle connecting to some singularities. Let us mention a well
known result in the following lemma, the proof is referred to Lemma
5.1 of \cite{LYXZ}.
\begin{lemma}\label{PHPCL1}
If potential $g$ of system \eqref{SN} is polynomial and satisfies
Assumption 1.1, in addition it exists some other isolated zeros,
then the period annulus is bounded. The period of center $O(0,0)$ is
divergent to infinity near the outer boundary.
\end{lemma}

If $h_s=+\infty$, we have the following conclusions \cite{DTHHZF}.

\begin{lemma}\label{PHPCL2}
Consider conservative second order autonomous duffing equation
$\ddot{x}+g(x)=0$, suppose that $g$ is locally Lipschitz and
$xg(x)>0$ in $(-\infty,+\infty).$ In addition if it satisfies
\begin{enumerate}
  \item [(i).]~ the superlinear condition $\lim\limits_{|x|\rightarrow\,+\infty}\frac{g(x)}{x}=+\infty,$
                 the system \eqref{SN} is said to be superlinear, then $\lim\limits_{h\rightarrow\,+\infty}T(h)=0$;
  \item [(ii).]~the sublinear condition $\lim\limits_{|x|\rightarrow\,+\infty}\frac{g(x)}{x}=0,$
                 the system \eqref{SN} is said to be sublinear, then $\lim\limits_{h\rightarrow\,+\infty}T(h)=+\infty$;
  \item [(iii).]~the semilinear condition $0<\liminf\limits_{|x|\rightarrow\,+\infty}\frac{g(x)}{x}\leq\,\limsup\limits_{|x|\rightarrow\,+\infty}\frac{g(x)}{x}<+\infty,$
                 the system \eqref{SN} is said to be semilinear, then
                 $0<\liminf\limits_{h\rightarrow\,+\infty}T(h)\leq\,\limsup\limits_{h\rightarrow\,+\infty}T(h)<+\infty.$
\end{enumerate}
\end{lemma}

In this respect we refer the reader to see Theorem 5.1 of
\cite{JVXZ} or Theorem 1.2 of \cite{LYXZ} for more details.

\section{The proof of Theorem \ref{PFCP}}\label{PHMR}
From the expression \eqref{EXPT1} of period function $T(x)$, since
$x\rightarrow\,\frac{1}{\sqrt{h-G(x)}}$ has singularities at
$x_{-}(h)$ and $x_{+}(h)$, it is necessary to make some
transformations. Following the change in coordinates applied in
\cite{LYXZ}, let
\begin{eqnarray}\label{TRANS}
G(x)&=&r\sin^{2k+2}\theta,\;\frac{y}{\sqrt{2}}=\sqrt{r}\cos\theta,\nonumber \\
x\sin\theta&\geq&\,0,\,\theta\in[-\frac{\pi}{2},\frac{\pi}{2}],\,r\in(0,h_s).
\end{eqnarray}
Under which the periodic orbit
\[\Gamma_h:\quad \frac{y^2}{2}+G(x)=h\]
is mapped to a cycle $r=h$. Along the cycle we have
\[g(x)dx=2(k+1)h\sin^{2k+1}\theta\cos\theta\,d\theta,\]
then the period function $T(h)$ yields to
\begin{equation}\label{EXP2}
T(h)=2(k+1)\sqrt{2h}\int_{-\frac{\pi}{2}}^{\frac{\pi}{2}}\frac{\sin^{2k+1}\theta}{g(x)\sqrt{1+\sin^2\theta+\cdots+\sin^{2k}\theta}}d\theta,
\end{equation}
where $x=x(\theta,h)$ is implicitly determined by the mapping
\eqref{TRANS}, recall that formula \eqref{EXP2} is proved as the
formula (3.11) of Proposition 3.2 in \cite{LYXZ}.

On the other hand, from the mapping \eqref{TRANS}, we get
\[\frac{\partial\,x}{\partial\,h}=\frac{\sin^{2k+2}\theta}{g(x)}=\frac{G(x)}{g(x)h}.\]
Hence we get $T^{\prime}(h)$ can be represented as follows, see also
formula (3.14) of Proposition 3.3 in \cite{LYXZ}.
\begin{equation}\label{EXPTD11}
T^{\prime}(h)=(k+1)\sqrt{\frac{2}{h}}\int_{-\frac{\pi}{2}}^{\frac{\pi}{2}}\frac{\delta(x)\sin^{2k+1}\theta}{\sqrt{1+\sin^2\theta+\cdots+\sin^{2k}\theta}}\,d\theta,
\end{equation}
where
\[\delta(x)=(\frac{G}{g^2})^{\prime}(x)=\frac{g^2-2Gg^{\prime}}{g^3}(x),\quad x\in\,(x_m,x_M).\]
Note that
\[G(x)=G(\sigma(x))=h\sin^{2k+2}\theta,\;x_{-}(h)<\sigma(x)<0<x<x_{+}(h), \]
and $x\sin\theta>0$ for
$\theta\in[-\frac{\pi}{2},\frac{\pi}{2}]\setminus\{0\}$ which
implies $x(-\theta,h)=\sigma(x(\theta,h))$. It yields to
\begin{eqnarray}\label{EXPTD12}
T^{\prime}(h)&=&(k+1)\sqrt{\frac{2}{h}}(\int_{-\frac{\pi}{2}}^{0}+\int_{0}^{\frac{\pi}{2}})\frac{\delta(x)\sin^{2k+1}\theta}{\sqrt{1+\sin^2\theta+\cdots+\sin^{2k}\theta}}\,d\theta, \nonumber\\
             &=&(k+1)\sqrt{\frac{2}{h}}\int_{0}^{\frac{\pi}{2}}\frac{\mathcal{B}_{\sigma}(\delta)(x)\sin^{2k+1}\theta}{\sqrt{1+\sin^2\theta+\cdots+\sin^{2k}\theta}}\,d\theta,
\end{eqnarray}
where $x\in (0,x_M)$ and $\mathcal{B}_{\sigma}(\delta)(x)$ is the
balance of function $\delta$ with respect to the involution
$\sigma$. From the mapping $G(x)=h\sin^{2k+2}\theta,$
differentiating both sides we get
\[g(x)dx=2(k+1)h\sin^{2k+1}\theta\cos\theta\,d\theta.\]
combining with \eqref{EXPTD12}, it brings to
\begin{equation}\label{EXPTD13}
T^{\prime}(h)=\frac{1}{\sqrt{2}h}\int_{0}^{x_{+}(h)}\mathcal{B}_{\sigma}(\delta)(x)\frac{g}{\sqrt{h-G}}(x)\,dx,
\; x_{+}(h)\in (0,x_M).
\end{equation}
The expression \eqref{EXPTD13} seems not to be new, while it is an
improvement of the expression (7) of \cite{XZZJ} based on Loud's
result \cite{WSL} where an additional condition is required that the
center is elementary.

Therefore from formula \eqref{EXPTD13}, the conclusions (1),(3) of
Theorem \ref{PFDISCR} are obvious. It is crucial to derive the
second conclusion about the (sharp) upper bound to the number of
critical periods.

If the center is elementary, $T^{\prime}(h)$ is continuous
in $(0,h_s)$ and can be continuously extended to $h=0$. Moreover
$\mathcal{B}_{\sigma}(\delta)(0)=0$.

If the center is nilpotent, from formula \eqref{DELNO} then we have
$\mathcal{B}_{\sigma}(\delta)(0)=\delta(0^+)-\delta(0^-)<0$.

Concerning the expression \eqref{EXPTD13} of $T^{\prime}(h)$, we
have $\frac{\partial\,x(h)}{\partial\,h}=\frac{1}{g(x)}>0$ in
$(0,x_M)$, it implies $G:=(0,x_M)\rightarrow\,(0,h_s)$ is a
differmorphism preserving orientation which does not change the number of zeros.
From expression \eqref{EXPTD13}, we get
\begin{equation}\label{EXPTD15}
\sqrt{2}hT^{\prime}(h)=\int_0^h\frac{\mathcal{B}_{\sigma}\delta(x^{-1}(u))}{\sqrt{h-u}}du
\end{equation}
where $x^{-1}(u)\in\,(0,x_M)$ is the inverse function determined by
$G(x)=u,\,u\in(0,h_s).$ Thus by Theorem \ref{PVD} we obtain the number
of zeros of function $\mathcal{B}_{\sigma}\delta(x^{-1}(u))$ in $(0,h_s)$ (therefore
$\mathcal{B}_{\sigma}(\delta)(x)$ in $(0,x_M)$ ) is an upper
bound to that of $T^{\prime}(h)$ in $(0,h_s)$. Moreover the bound is achievable if function
\[\int_0^h\mathcal{B}_{\sigma}\delta(x^{-1}(u))du=\int_{x_-(h)}^{x_+(h)}\delta(x)dx=\frac{G}{g^2}(x)-\frac{G}{g^2}(z)\]
also has $n$ zeros in $(0,x_M)$ where $G(x)=G(z),\,x_m\leq\,z<0<x\leq\,x_M.$ It completes the proof of conclusion (2) of Theorem
\ref{PFCP}.

To assert that the period function of center is monotone indeed (absence of critical periods), one sufficient condition on monotonicity of period through verifying the
convexity of $T(h)$ is helpful in many literatures \cite{MS,JVXZ}. Similarly as above deductions, we get
the derivative of $T(h)$ of order two as follows
\begin{equation}\label{EXPTD2}
T^{\prime\prime}(h)=\frac{1}{\sqrt{2}h^{2}}\int_{0}^{x_{+}(h)}\mathcal{B}_{\sigma}(\phi)(x)\frac{g}{\sqrt{h-G}}(x)\,dx,
\; x_{+}(h)\in (0,x_M),
\end{equation}
where
\[\phi(x)=\delta^{\prime}(x)\frac{G}{g}(x)-\frac{\delta(x)}{2}=\frac{12(Gg^{\prime})^2-4gG(g^{\prime\prime}G+g^{\prime}g)-g^4}{2g^5}(x).\]

{\bf Proof of Theorem \ref{NCPSN}}
We will prove in the sequel by applying Rolle's Theorem and its
generalization.

Firstly we generalize Lemma 4.5 of \cite{LYXZ} to the following
\begin{lemma}\label{PROPg}
~Assume that $g(x)$ is a polynomial of degree no less than 2, all
its zeros are real and satisfies the condition \eqref{PROPG}, then
\[\delta\,^{\prime}(x)>0,\quad\, x\in\,(x_m,0)\cup(0,x_{M}).\]
\end{lemma}

\begin{proof}
(1)~Suppose that the potential $g(x)$ is an odd polynomial with
degree no less than three, recall that $g(0)=0$, that follows $g(x)$
has at most $\frac{deg(g)-1}{2}$ positive zeros. From Theorem
\ref{PFCP}(2) we get the number of critical periods for system
\eqref{SN} is bounded by the number of positive zeros of function
\[\delta(x)=\Big(\frac{G}{g^2}\Big)^{\prime}(x)=\frac{g^2-2Gg^{\prime}}{g^3}(x).\]
at $(0,\,x_M)$. Let $N(x)=(g^2-2Gg^{\prime})(x)$ be the numerator of
function $\delta(x)$, note that $g(x)>0$ at $(0,x_M)$. Then
$N^{\prime}(x)=-2Gg^{\prime\prime}(x)$, recall that $G(x)$ is
positive at $(0,\,x_M)$ and $N(0)=0$ holds, thus the number of zeros
of function $N^{\prime}(x)$ is no more than $\frac{deg(g)-3}{2}$ at
$(0,\,x_M)$, that follows function $\delta(x)$ has at most
$\frac{deg(g)-3}{2}$ zeros at $(0,x_M)$.

(2)~By using Generalized Rolle's Theorem \ref{GRTHEM}, the number of
solutions of SAS
\[\sharp\{(x,z)\in\,(0,x_M)\times\,(x_m,0)\,|\,G(x)=G(z),\;\delta(x)=\delta(z)\}\]
is no more than one plus the number of the solutions of the SAS
\[\sharp\{(x,z)\in\,(0,x_M)\times\,(x_m,0)\,|\,G(x)=G(z),\;\delta^{\prime}(x)g(z)=\delta^{\prime}(z)g(x)\}.\]

Since for all $x\in(x_m,0)$ and $(0,X_M)$ there holds $xg(x)>0$ and $\delta^{\prime}(x)>0$ deduced by Lemma \ref{PROPg}. Then from Theorem \ref{GRTHEM} we assert the former SAS \eqref{SAS} has at most one zero, therefore Theorem
\ref{PFCP}(2) concludes that period function $T(h)$ has at most one critical points.
\end{proof}

\section{Applications}\label{PFRZ}
In this section we give several applications of Theorem \ref{PFCP}.
By means of analytical and computational techniques, the analytical
behaviors of period function are proved.

Firstly to illustrate our approaches, we propose more simpler proofs
to two known results.

{\bf Example \ref{PFRZ}.1} We consider the family of
potential system with even potential energy
\begin{equation}\label{HS1}
\begin{split}
\frac{dx}{dt}=-y, \quad \frac{dy}{dt}&=x+kx^2+x^5.
\end{split}
\end{equation}

For $k\in(-2,+\infty)$ the system \eqref{HS1} is superlinear and the origin $O(0,0)$ is a global center.
In Theorem 1.1(b) of \cite{FMJV} and Theorem 1.5 of \cite{JGAG}, it is proved that the system (5.1)
associated with the global center has at most one critical period. By using the Harmonic Balance Method Theorem 1.5 of \cite{JGAG} is stated as follows
\begin{itemize}
  \item[(i)] The function is monotonous decreasing for $k\geq\,0$.
  \item[(ii)] System \eqref{HS1} has exactly one critical period for $k\in\,(-2,0)$, the
period function starts increasing, until a maximum (a critical period) and then decreases
towards zero.
\end{itemize}

From Lemma \ref{PHDPHO} and \ref{PHPCL2} we have
\begin{equation}\label{EQN}
\lim\limits_{h\rightarrow\,0^+}T(h)=2\pi,\quad \lim\limits_{h\rightarrow\,0^+}T^{\prime}(h)=-\frac{3k\pi}{2}, \quad \lim\limits_{h\rightarrow\,+\infty}T(h)=0.
\end{equation}

An easy computation shows for all $x\in(0,+\infty)$,
\[\mathcal{B}_{\sigma}(\delta)(x)=2\delta(x)=-\frac{x(4x^6+9kx^4+3k^2x^2+20x^2+9k)}{3(x^4+kx^2+1)^3}.\]
It is obvious for $k\geq\,0$ Theorem \ref{PFCP}(1) implies that period function of system
\eqref{HS1} is monotone decreasing globally.

As $-2<k<0$ we get from the expressions \eqref{EQN} that system \eqref{HS1} has at least one critical period.
In what follows we intend to prove the assertion that the funtion $\mathcal{B}_{\sigma}(\delta)(x)$ has exactly one
positive zero. By change of variable $x^2\rightarrow\,u$ the sextic polynomial $4x^6+9kx^4+3k^2x^2+20x^2+9k$ is changed to
cubic polynomial with parameter $k\in(-2,0)$
\[\omega(u,k)=4u^3+9ku^2+(3k^2+20)u+9k.\]
By Intermediate-Value Theorem it has at least one positive zero. Furthermore by Descartes'
law of sign it might even have three positive zeros (multiplicities taken into account). Differentiating function
$\omega(u,k)$ with respect to $u$ we get
\[\frac{\partial \omega(u,k)}{\partial u}=12u^2+18ku+(3k^2+20),\]
which does not vanish at $(0,\infty)$ for all $k\in\,(-2,0).$ The contradiction confirms our assertion.
Hence from Theorem \ref{PFCP}(2) system \eqref{HS1} has exactly one critical period for $k\in\,(-2,0)$ which is a maximum.

{\bf Example \ref{PFRZ}.2}~We consider a potential system with
isochronous center where the potential is not polynomial.
\begin{equation}\label{HS3}
\begin{split}
\frac{dx}{dt}=-y, \quad
\frac{dy}{dt}&=(1+\frac{x}{4})-(1+\frac{x}{4})^{-3}.
\end{split}
\end{equation}

By direct computation we get the potential energy
\[G(x)=(x+\frac{x^2}{8})+2(1+\frac{x}{4})^{-2}-2.\]
The origin $O(0,0)$ is nondegenerate center surrounded by an
unbounded period annulus whose projection onto $x$-axis is interval
$(-4,+\infty).$ By direct computation, we get
\[\delta(x)=(\frac{G}{g^2})^{\prime}(x)=\frac{128(4+x^3)}{(x^2+8x+32)^3}.\]
Since $xg(x)>0$ holds for $x\in(-4,+\infty)\setminus\,\{0\}$, there
exists an analytic involution $\sigma$ satisfying
$z=\sigma(x)\in(-4,0)$ and $G(x)=G(z).$ It is easy to verify that
Semi-algebraic system \eqref{SAS} always have solutions for all
$(x,z)\in(0,\infty)\times(-4,0)$. As a matter of fact the equations
$G(x)=G(z)$ and $\delta(x)=\delta(z)$ have common factor $xz+4(x+z)$
excluding $x-z$. It implies $\mathcal{B}_{\sigma}(\delta)(x)$
identically equals to zero for all $x\in(0,+\infty)$, thus from
Theorem \ref{PFCP}(3) we assert the origin is an isochronous center.

Note that system \eqref{HS3} can be changed to the so-called
dehomogenized Loud's system with parameters $(D,F)=(0,\frac{1}{4})$
by some coordinate transformations, as proposed in Lemma 2.3 of \cite{FMJV2}.
It is known that the origin is one of four Loud's isochronous
centers.

It is remarked in \cite{CAMFVJ} (pp 382--383) that perhaps the
easiest families of potential Hamiltonian other than the families of
polynomial function that one could study next are those potential
energy $G$ is an entire function or a rational function without real
poles. An interesting question for further research is to construct
the isochronous center in these families differ from the linear one.

{\bf Example \ref{PFRZ}.3.} As the applications of Theorem \ref{NCPSN} and Lemma \ref{PROPg}, we consider system \eqref{SN} in case of all the zeros of the potential $g$ are real.

(1) When $g(x)=x(1-x^2)$, the origin of system \eqref{SN} is an elementary center, test function $\delta^{\prime}(x)>0$ for all $x\in\,(-1,1)$, thus the system has a monotone increasing period function, i.e. absence of critical periods.

(2) when $g(x)=x^3(1-x^2)$ or $x^3(x+1)$, the origin of system \eqref{SN} is nilpotent center, test function $\delta^{\prime}(x)>0$ for all $x\in\,(x_m,0)$ or $(0,x_M)$, the system has exactly one critical period which is a minimum.

In what follows, we consider the hyperelliptic Hamiltonian system of
degree five
\begin{equation}\label{HSLGID}
\begin{split}
\frac{dx}{dt}=-y, \quad \frac{dy}{dt}&=x(x+1)(x^2+\beta\,x+\alpha)
\end{split}
\end{equation}
with real parameters $\alpha,\,\beta$. The normal form of
hyperelliptic Hamiltonian of degree five was proposed in
\cite{GLID}. As $\beta^2-4\alpha\geq\,0$, it was proved in
\cite{CLKL, LYXZ} the period function of any period annulus of the
system has at most one critical point, and it has exactly one if and
only if the period annulus surrounds three equilibria, counted with
multiplicities. The proof in latter work is purely analytical,
different from that computer
algebra is utilized in the former work. Moreover it is conjectured in \cite{CLKL} the system
\eqref{HSLGID} has at most two critical periods. As
$\beta^2-4\alpha<0$, system \eqref{HSLGID} has a pair of conjugated
complex critical points, it has only one phase portrait
topologically: the period annulus surrounds a non-degenerate center
$O(0,0)$ and terminates at a homoclinic loop connecting to a
hyperbolic saddle $S(-1,0)$. In this case the period function may be
monotone, or have one or two critical points, depending on the
values of parameters $\alpha,\,\beta$.

The phrase portrait of system \eqref{HSLGID} with
$\beta^2-4\alpha<0$ is sketched as Fig.1 of \cite{JW1}.

Noticing that when polynomial potential is asymmetric, the involution $\sigma$ in
general can not be explicitly stated, the problem on critical period
can be reduced to solving (constant or parameteric) Semi-algebraic system.

Recall that the conclusion is known in \cite{CLKL, LYXZ} where the
period function of potential system \eqref{HSLGID} with
$\beta^2-4\alpha\geq\,0$ was tackled by different approaches. The
proofs therein are long and highly nontrivial.

{\bf Example \ref{PFRZ}.4.} In literature there are few results concerning the existence of exact two critical period orbits.
we have the following conclusion.
\begin{theorem}
when $\beta=-1,\,\alpha=\frac{1}{2}$ the system \eqref{HSLGID} has exactly two critical
periods.
\end{theorem}
\begin{proof}
When $\beta=-1,\,\alpha=\frac{1}{2}$, the system \eqref{HSLGID} has a pair of complex
critical points. By Lemma \ref{PHDPHO}, we get near the elementary
center $O$ it holds
\[T^{\prime}(0^+)=\frac{\pi}{6}\frac{(10\alpha^2+11\alpha\beta+10\beta^2-9\alpha)}{\alpha^{7/2}}.\]
In this case we have $T^{\prime}(0^+)>0$. Note that
$T(h_s^-)=+\infty.$ Hence $T(h)$ can be monotone increasing or have
even number of critical periods in $(0,h_s)$, counted with
multiplicities.

Recall that the author of
present paper in \cite{JW1} considered how to bound the number of
limit cycles bifurcating from the period annulus under small
polynomial perturbations, whether the system \eqref{HSLGID} with $\beta=-1,\,\alpha=\frac{1}{2}$ has two
critical periods is left open.

By an easy computation, we get $x_M\approx\,0.9239964237.$ By using
Maple 18, we get
\[\frac{\delta(x)-\delta(z)}{x-z}=\Psi(x,z),\]
where $\Psi(x,z)$ is a symmetric polynomial of $x,z$ with degree 13.
By command {\bf RealRootIsolate} in Maple 18 we assert the equations
$U(x,z)$ and $\Psi(x,z)$ in the domain $(0,x_M)\times(-1,0)$ have
two isolated zeros contained in the following domains
\[(0.3560526240,-0.2935057702)\in\,[\frac{25}{128}, \frac{51}{256}]\times\,[-\frac{91}{128}, -\frac{181}{256}]\]
and
\[(0.7682670211,-0.6425079942)\in\,[\frac{51}{256}, \frac{13}{64}]\times\,[-\frac{229}{256},
-\frac{57}{64}],\] respectively. Note that
$x_N^1\approx\,0.3560526240$ and $x_N^2\approx\,0.7682670211$ are
all nodal zeros of function $\mathcal{B}_{\sigma}(\delta)(x)$ in
$(0,x_M)$. Hence from Theorem \ref{PFCP}(2), the system
\eqref{HSLGID} can have at most two critical periods.

We can refine the conclusion by excluding the period function is monotonic.
On account of Lemma \ref{MFNC}, if the center of system \eqref{HSLGID} at the origin has a monotone period function, then the  the test function $P(x,z)=\frac{G(x)}{(x-z)^2},\,z=\sigma(x)$ should be monotone in $(0,x_M)$ as well. With the aid of Maple we obtain an negative result through verifying that its total derivative with respect to $x$, namely $P_x(x,z)=\frac{\partial\,P}{\partial\,x}+\frac{\partial\,P}{\partial\,z}\frac{dz}{dx}$ has two common zeros with $G(x)=G(z)$ at domain $(x,z)\in\,(0,X_M)\times(-1,0).$
Therefore we assert that when $\beta=-1,\,\alpha=\frac{1}{2}$ the center of system \eqref{HSLGID} at the origin has exactly two critical periods.
\end{proof}

\section{Remarks and Discussions}\label{DISS}
We remark here that Theorem 1.1(2) proposes an efficient criterion
for bounding the number of critical periods of potential center, while how to obtain the sharp bound is much more difficult.

Consider the critical periods of the center of Hyperelliptic Hamiltonian system \eqref{HSLGID} of
degree 5, we leave {\it Some unsolved cases}. For instance when $\beta=\frac{7}{5}$ or $-\frac{7}{5}$ and $\alpha=\frac{1}{2}$ an easy computation lead to $T^{\prime}(0^+)>0$ and
$T(h_s^-)=+\infty$ as well. Whether the period function of center at the origin is monotone increasing or has two critical points is left open. The trials and try fails we can assert by Theorem 1.1(2) it has at most two critical periods. While we can not dismiss the claims that the period function is monotone increasing as the mentioned above case $\beta=-1,\,\alpha=\frac{1}{2}$ by Lemma \ref{MFNC} for the test function $P(x,z)$ is monotone in $(0,X_M)$. The effort to confirm monotonicity of period has also defeated by verifying the convexity of period function through expression \eqref{EXPTD2}, unfortunately in these cases test function $\mathcal{B}_{\sigma}(\phi)(x)$ has two zeros in $(0,X_M)$.

How to derive more efficient and convenient approaches to determine the sharp bound to
the number of critical periods, it is an important and challenging
topic, we leave it for further research.

%For acknowledgements section, please don't number the section, please begin it with \section*{Acknowledgements}
\section*{Acknowledgements}
The author would like to thank the anonymous referees for their
helpful corrections and valuable suggestions which improve the
presentation of this paper.

% You may incorporate your references as follows in your main tex file.
% Using BibTex is not recommended but can be handled.

\end{document}